\documentclass[12pt]{article}  

\newcommand{\al}{\alpha}
\newcommand{\var}{\varphi}

\newcommand{\ti}{\times}

\newcommand{\del}{\delta}
\newcommand{\pa}{\partial}
\newcommand{\f}{\frac}

\newcommand{\subs}{\subseteq}

\newcommand{\rig}{\rightarrow}
\newcommand{\lo}{\longrightarrow}

\input{epsf}

    \usepackage{graphicx}
    \usepackage{amsmath}
    \usepackage{amsthm}
    \usepackage{amssymb}
    \usepackage{latexsym}

\newtheorem{theorem}{Theorem}[section]
\newtheorem{lemma}[theorem]{Lemma}

\theoremstyle{definition}

\newtheorem{example}[theorem]{Example}
\newtheorem{remark}[theorem]{Remark}

\title{Generalizations of the Hilbert-Weierstrass Theorem
and Tonelli-Morrey Theorem: the Regularity of Solutions of Differential Equations and Optimal Control Problems}
\author{Saman Khoramian\thanks{E-mail address: saman.khoramian@gmail.com}}
\date{ }

\begin{document}

\maketitle

\begin{abstract}
\noindent
One of the basic problems in the “Calculus of Variations” is minimizing the following integral function:\\ 
                  $$F(x)=\int_a^b f(t,x(t),x'(t)) dt$$
over a class of functions $x$ defined on the interval $[a,b]$, and which take prescribed values at $a$ and $b$. Solutions of this basic problem under a regularity theorem, lie in a smaller class of more regular functions, however, they were initially considered to lie merely in a larger class. Two Theorems attributed to “Hilbert-Weierstrass” and “Tonelli-Morrey” respectively are two classical studies for the regularity discussion around the solutions of this problem. Now, since differential equations and optimal control problems with higher-order have been growing in the literature, addressing the regularity issues for these problems should be paid more attention. In this regard, here, a generalization for the regularity theorems will be presented; namely, the regularity of the solution of the following integral functional \\
                             $$F(x)=\int_a^b f(t,x(t),x'(t),\dots,x^{(n-1)}(t)) dt$$
where $n \geq 2$. It is desired that these theorems will be useful for researchers to prove the regularity properties of differential equations or optimal control problems.  
\end{abstract}

\section{Introduction} 
“Differential equations”, compared with other mathematical tools, have a more paramount role in explaining how the physical world functions. Systems of ordinary differential equations of the form\\ 
\begin{equation}
F(x,y,y',\ldots ,y^{(n-1)})=y^{(n)}\tag{$1-1$}
\end{equation}
are routinely used today to model a wide range of phenomena, in areas as diverse as aeronautics, power generation, robotics, economic growth, and natural resources. For solving $(1-1)$, an $y \in C^{n}$ should be found that comes true in the equation. In some special cases, there exist direct methods for achieving the exact solution (such as Bernoulli equations and so on), which are usually addressed in elementary differential equations books (see [16], [17]; for example). However, in real-world applications, most differential equations have more complicated forms; consequently, in order to approach this kind of problems many different numerical methods for approximating the solutions exist. Alongside that, there exist theoretical approaches that address the existence and number of solutions as well as analysis around the properties of the solutions for differential equations. These theoretical efforts light the way for the numerical tasks. \\
The approach of theoretical works around the existence of solutions mostly is followed through investigating the solution in a larger space than $C^n$; namely, $W^{{n-1},{2}}$ that is the space of all $n-1$ weakly derivative functions in $L^2$. In fact, the problem of finding a solution in the space $C^n$ is replaced with the problem of finding a weak solution in the $W^{{n-1},{2}}$ that is reflexive and a much larger space than $C^n$; consequently more achievable for investigating the problem of existence of solutions in the point of utilizing mathematical analysis theorems. After the proof of the existence of a weak solution in $W^{{n-1},{2}}$, what remains is to prove that the weak solution belongs to $C^n$; this is named the regularity of the weak solution. Also, note that proving the weak solution in $W^{{n-1},{2}}$ which belongs to $C^{n-1}$ is named regularity of the classical solution.\\ 
From the beginning of differential equation theory, differential equations of order $2$ have been more important. In other words, the Mathematical interpretation of most applied problems in physics and engineering are the quadratic differential equations. Therefore, the regularity of solutions to these kinds of differential equations has been at the center of attention in the literature.  Among these efforts, Theorems 7.1.13 and 7.1.14 in [7] could be mentioned that are for the regularity of the classical solution and the regularity of the weak solution of a category of quadratic ordinary differential equations.  Theorem 7.1.13 is a classical result of Hilbert and Weierstrass around 1875 (in the lecture notes of Weierstrass that were circulating then) and has appeared in countless books in the calculus of variations since then. For instance, you can find another form of it in [4] (see Theorem 15.7). Goldstein’s 1980 book “a history of the calculus of variations,” says something about it [9]. Theorem 7.1.14 is a version of the so-called “Tonelli-Morrey” approach to regularity in the literature. It dates back to Tonelli’s pioneering book of 1921 “Fondamenti del calcolo delle variazioni” [18] and Morrey’s book of 1966 [14]. A relevant discussion appears in chapter 16 of [4]: To recover Theorem 7.1.14, one uses Theorem 16.13 to get Lipschitz regularity first (see also the remark at the bottom of page 329), then theorem 15.5 to get $C^{1}$, then theorem 15.7 for the higher regularity. \\  
The approach in these theorems is considering {\it energy functional} related to the concerning differential equation and applying the fact that the extreme of the energy functional is a weak solution for the differential equation and vice versa. In fact, it has been shown that if $x$ is an extreme point of the energy functional, then $x \in C^{2}$. \\
Since using the differential equations of the order of larger than $2$ have been increasing by various interpretations in different practical problems (see [10, 12, 13, 20], for example), theoretical discussions around these kinds of differential equations should be addressed to a greater extent. In this paper, we present the general form of the ordinary differential equations of order $n \geq 2$ and prove the regularity of the weak solutions by assuming the existence of some solutions. Indeed, the approach that has been provided in Theorem 7.1.13 and Theorem 7.1.14 in [7] will be generalized. As an explanation, suppose the general form of the boundary value problems for the $n$-th order ordinary differential equations can be expressed as the following form:
\begin{equation*}
\begin{cases}
G(t,x(t),x'(t),\dots,x^{(n)}(t))=0, & t\in (a,b),\\
x^{(i)}(a)=u_i & \text{for} \quad i\in N,\\
x^{(j)}(b)=w_j & \text{for} \quad j\in N',
\end{cases}\tag{$1-2$}
\end{equation*}
where $x^{(i)}$ is the $i$-th derivative of the function $x$, $u_i,w_j\in \mathbb{R}$ for 
$i\in N, j\in N'$, $n\geq 2$ and $N,N'\subseteq \{0,1,\dots,n-1\}$. We are interested in proving regularity results for solutions of the problem $(1-2)$.  The existence of solutions and some regularity results for $n=1$ have been discussed by mathematicians in advance. As already mentioned, since applied interpretations for ODE’s with degrees more than $2$ have come to the fore, papers for investigating the problem of the existence of solutions of these equations are increasing. Therefore, in parallel to these papers, for regularity results of them, some efforts should be made. In fact, what here is done is providing the regularity results for the solution of the problem $(1-2)$. To this aim, we consider this fact that any critical point of the following energy functional $F$ is a weak 
solution of $(1-2)$ and vice versa:
\begin{equation}
F(x)=\int_a^b f(t,x(t),x'(t),\dots,x^{(n-1)}(t)) dt,\quad x\in {\mathcal{N} }_{N,N'} \tag{$1-3$}
\end{equation}
where ${\mathcal{N}}_{N,N'}=\{u\in W^{n-1,2}(a,b): u^{(i)}(a)=u_i, u^{(j)}(b)=w_j 
\;\; \text{where}\;\; i\in N, j\in N'\}$
and $f=f(x_1,\dots,x_{n+1})$ is a function defined on $[a,b]\ti \mathbb{R}^n$ with 
continuous second partial derivatives with respect to all its variables.

Our approach for regularity in this paper is to show that, if $u_0$ 
is a critical point of $F$, then $u_0\in C^n[a,b]$. This work has been done for 
$n=2$ and $N,N'=\{0\}$ (see Theorem 7.1.13 and 7.1.14 from [7]). Here, we 
prove general case; $n\geq 2$ and $N,N'\subseteq \{0,1,\dots,n-1\}$.\\
In Section 2, the “regularity of the classical solutions” for the function $F$ in $(1-3)$ is presented. In Section 3, it will be shown that if $u$ is a local extremum of $F$ in $(1-3)$ with respect to $W^{{n-1},{2}}$, then $u$ is in $C^n$.

\section{Regularity of the classical solution} 
To be able to infer that the solution of the notation map of $(1-2)$ not only is of the class $C^{n-1}$ but also is in $C^{n}$ of initial data, we must place some requirements. We begin with the following Lemma:  
\begin{lemma}
Suppose $f$ is a function from $\mathbb{R}^{n+1}$ to $\mathbb{R}$ and its partial derivatives exist. Then we will have
\begin{align*}
\lim_{r \to 0} &\f {f(x_0, x_1+rm_1, x_2+rm_2, \dots, x_n+rm_n) - f(x_0, x_1, x_2, \dots, x_n)}{r}\\
&= \sum_{i=1}^{n} m_{i}\f{\pa f}{\pa x_i} (x_0, x_1,\dots, x_n).
\end{align*}
\end{lemma}
\begin{proof} 
First, note the following equality 
\begin{align*}
&f(x_0, x_1+rm_1, x_2+rm_2, \dots, x_n+rm_n) - f(x_0, x_1, x_2, \dots, x_n)\\
=&f(x_0, x_1+rm_1, x_2, \dots, x_n) - f(x_0, x_1, x_2, \dots, x_n)\\
+&\sum_{i=2}^{n} [f(x_0, x_1+rm_1, \dots, x_i+rm_i, x_{i+1}, x_{i+2}, \dots, x_n)\\
&~~~~~- f(x_0, x_1+rm_1, \dots, x_{i-1}+rm_{i-1}, x_i, x_{i+1}, \dots, x_n)].
\end{align*}
Then, by this the proof comes from the following fact:\\
\\ 
Fact: {\textit {Suppose $j \in \lbrace 1, \dots, n\rbrace$ and $m_i, t_i \in \mathbb{R};~1\leq i \leq n$.
If $m_i = t_i$ for $i \neq j$ and $m_j \neq t_j = 0$, then
\begin{align*}
\lim_{r \to 0} &\f {f(x_0, x_1+rm_1, \dots, x_n+rm_n) - f(x_0, x_1+rt_1, \dots, x_n+rt_n)}{r}\\
&= m_{i}\f{\pa f}{\pa x_i} (x_0, x_1,\dots, x_n).
\end{align*}}} 
\end{proof}
Moreover, the following lemma, which is a simple result of the Implicit Function Theorem [15], is needed: 
\begin{lemma}
If $\var:=\var(s,t)$ is a function from $[a,b]\ti \mathbb{R}$ to $\mathbb{R}$ 
such that 

(i) $\var(t_0,s_0)=0$.

(ii) $\f{\pa \var}{\pa s}(t_0,s_0)\neq 0$.

(iii) $\var,\f{\pa\var}{\pa s}$ are continuous in $t_0$.\\
Then,
$$\exists \del_1,\hat{\del}; \forall t\in (t_0-\del_1,t_0+\del_1) \exists ! 
z(t) \in (s_0-\hat{\del}, s_0+\hat{\del}); \var(t,z(t))=0,$$
the function $t\lo z(t)$ is continuous.
\end{lemma}
We recall the Fundamental Lemma in Calculus of Variation by du Bois-Reymond [8](see Lemma 7.1.9 in [7]): 
\begin{lemma}
Let $\mathcal{I}$
be an open interval and $f\in L^1_{loc}(\mathcal{I})$. If 
$$\int_{\mathcal{I}} f(x) \var'(x)dx=0 \quad \text{for any} \quad \var\in 
C_0^\infty({\mathcal{I}}),$$
then $f=const$. a.e. in ${\mathcal{I}}$.
\end{lemma}
We have provided a generalization for it in the following lemma: 
\begin{lemma}
Let ${\mathcal{M}}=\{u\in C^n[a,b]: u^{(i)}(a)=u_i, u^{(i)}(b)=w_i;~ 0\leq i\leq n\}$
and $f\in L^1_{loc}(a,b)$. If
$$\int_a^b f(t) V^{(n)}(t) dt=c \quad \text{for any}\quad V\in{\mathcal{M}}$$
then $f$ is a polynomial of degree $n$ almost everywhere in $[a,b]$ .i.e. there are $c_0,c_1, \dots,c_n \in \mathbb{R}$ such that:
$$f(t)=c_nt^n+\dots+c_1t+c_0\quad \text{ a.e.\; in} \quad [a,b].$$
\end{lemma}

\begin{proof}
First, we define ${\mathcal{M}}'$ as follows:
$${\mathcal{M}}'=\{u\in C^n[a,b]: u^{(i)}(a)=u^{(i)}(b)=0;~ 0\leq i\leq n\}.$$
Moreover, we define $V_0$ and $G(\cdot;A,B)$ as follows:
\begin{align*}
V_0(t) & =\f{w_0-u_0}{b-a} (t-a) +u_0 , \\
G(t;A,B) &=\f{B-A}{b-a} (t-a)+A.
\end{align*}
Then, by defining $V_n$ iteratively as follows, 
$$V_n(t) =V_{n-1}(t)+ G_n(t) (V_0(t)-u_0)^n (V_0(t)-w_0)^n$$
where $G_n(t)=G(t;A_n,B_n)$ such that:
\begin{align*}
A_n &=(-1)^n \f{(b-a)^n}{n!(w_0-u_0)^{2n}} (u_n-V_{n-1}^{(n)}(a)) \\
B_n &= \f{(b-a)^n}{n!(w_0-u_0)^{2n}} (w_n-V_{n-1}^{(n)}(b)).
\end{align*}
we will have $V_n+{\mathcal{M}}'={\mathcal{M}}$. Therefore,
$$\forall V\in {\mathcal{M}}'~~\int_a^b f(t) [V^{(n)}(t)+V_n^{(n)}(t)] dt=c.$$
Then,
$$\forall V\in{\mathcal{M}}'~~\int_a^b f(t) V^{(n)}(t)dt=c-\int_a^b f(t) V_n^{(n)} 
(t)dt:=c'.$$
Since $\al{\mathcal{M}}'={\mathcal{M}}'$ for $\al\neq 0$, we will have
\begin{align*}
\forall & V \in {\mathcal{M}}'~~ \int_a^b f(t) V^{(n)}(t)dt=\f{c'}{2} 
\end{align*}
Therefore $\f{c'}{2}=c'$ i.e. $c'=0$. We get, by integrating by parts iteratively, 
\begin{align*}
\forall & V \in {\mathcal{M}}'~~ \int_a^b f^{(n-1)}(t) V'(t)dt=0. 
\end{align*}
Consequently, since $C_0^\infty (a,b)\subs{\mathcal{M}}'$, by Lemma 2.3, we will have $f^{(n-1)}=const$. a.e. in $[a,b]$ then $f$ is a polynomial of degree $n$ almost everywhere in $[a,b]$.
\end{proof}

The following regularity theorem is our major goal in this section.

\begin{theorem}
Suppose $n\geq 1$, $N,N'\subs\{0,1,2,\dots,n\}$,
and ${\mathcal{M}}_{N,N'}=\{u\in C^n[a,b]: 
u^{(i)}(a)=u_i, u^{(j)}(b)=w_j \; where \; i\in N, j\in N'\}$.
Define the functional $F$ on ${\mathcal{M}}_{N,N'}$ by 
$$F(u)=\int_a^b f(t,u(t),u'(t),\dots, u^{(n)}(t))dt$$
where $f=f(x_1,\dots,x_{n+2})$ is a function defined on $[a,b]\ti \mathbb{R}^n$ with 
continuous second partial derivatives with respect to all its variables. Let 
$u_0\in{\mathcal{M}}_{N,N'}$ be a local extremum of $F$ with respect to ${\mathcal{M}}_{N,N'}$, 
and let $t_0\in (a,b)$ be such that 
$$\f{\pa^2 f}{\pa x^2_{n+2}} (t_0,u_0(t_0),u'_0(t_0),\dots,u_0^{(n)}(t_0))\neq 
0.$$
Then there exists $\del>0$ such that $u_0\in C^{n+1}(t_0-\del,t_0+\del)$. 
\end{theorem}

\begin{proof}
Let $V\in{\mathcal{M}}_{N,N'}$. Then by Lemma 2.1,
\begin{align*}
&\del F(u_0;V) =\lim_{r\rig 0} \f{F(u_0+rV)-F(u_0)}{r}\\
&=\int_a^b \lim_{r\rig 0} \f{f(t,u_0(t)+rV(t), 
u'_0(t)+rV'(t),\dots,u_0^{(n)}(t)+rV^{(n)}(t)) - 
f(t,u_0(t),\dots,u_0^{(n)}(t))}{r} dt \\
&=\int_a^b \sum_{i=2}^{n+2} \f{\pa f}{\pa x_i} (t,u_0(t),\dots, u_0^{(n)}(t)) 
V^{(i-2)} (t)dt.
\end{align*}
Therefore, by Euler Necessary Condition,  
$$\del F( u_0; V)=0 \quad \text{for}\quad V\in{\mathcal{M}}_{N,N'}.$$
Consequently, 
$$\sum_{i=2}^{n+2} \int_a^b \f{\pa f}{\pa x_i} (t,u_0(t),\dots, 
u_0^{(n)}(t))V^{(i-2)}(t)dt=0\; ~~for\; ~~V\in{\mathcal{M}}_{N,N'}.$$
Define $h_{0,j}(t):=\f{\pa f}{\pa x_j}(t,u_0(t), \dots, u_0^{(n)}(t))$ 
for $2\leq j\leq n+2$ and $h_{k,j}$ iteratively as follows:
$$h_{k,j} (t):= \int_a^t h_{k-1,j}(\xi) d\xi \quad for \quad k\geq 1, 2\leq 
j\leq n+2.$$
Integrating by parts iteratively implies that
\begin{multline*}
\int_a^b \f{\pa f}{\pa x_{n-m+2}} (t,u_0(t), \dots, u_0^{(n)}(t)) V^{(n-m)} 
(t)dt =\sum_{j=1}^m (-1)^{j+1} h_{j,n-m+2} (b) V^{(j+n-m-1)}(b) \\
+(-1)^m \int_a^b h_{m,n-m+2}(t) V^{(n)}(t) dt \quad for \quad 1\leq m\leq n.
\end{multline*}
Consequently, Since 
${\mathcal{M}}:={\mathcal{M}}_{\{0,1,\dots,n\},\{0,1,\dots,n\}}
\subs {\mathcal{M}}_{N,N'}$, for any $V\in{\mathcal{M}}$ 
\begin{align*}
\int_a^b h_{0,n+2} (t) V^{(n)}(t)dt 
+\sum_{m=1}^n \Biggl( \sum_{j=1}^m& (-1)^{j+1} h_{j,n-m+2}(b) V^{(j+n-m-1)}(b)\\
&+(-1)^m \int_a^b h_{m,n-m+2}(t) V^{(n)}(t)dt \Biggr) =0
\end{align*}
thus,
$$\sum_{m=1}^n \left( \sum_{j=1}^m (-1)^{j+1} h_{j,n-m+2}(b) V^{(j+n-m-1)}(b) 
\right) +\sum_{m=0}^n \left( (-1)^m \int_a^b h_{m,n-m+2}(t) V^{(n)}(t)dt 
\right) =0.$$
Let
$c:=-\sum_{m=1}^n \left( \sum_{j=1}^m (-1)^{j+1} h_{j,n-m+2} (b) 
V^{(j+n-m-1)}(b) \right)$. Hence, we have
$$\forall V\in {\mathcal{M}} \int_a^b \left( \sum_{m=0}^n (-1)^m 
h_{m,n-m+2}(t)\right) V^{(n)}(t) dt=c.$$
By Lemma 2.4 we 
conclude there are $c_0,c_1, \dots,c_n \in \mathbb{R}$ such that: 
$$\sum_{m=0}^n (-1)^m h_{m,n-m+2} (t)=c_nt^n+\dots+c_1t+c_0\quad \text{ a.e.\; in} \quad [a,b].$$
Since $u_0\in C^n[a,b]$, $u_0^{(n)}$ is continuous and consequently 
$$\f{\pa f}{\pa x_{n+2}}(t,u_0(t), \dots, u_0^{(n)}(t)) +\sum_{m=1}^n (-1)^m 
h_{m,n-m+2}(t)=c_nt^n+\dots+c_1t+c_0$$
for all $t\in [a,b]$. For $t\in [a,b]$ and $s\in \mathbb{R}$ define a function $\var$ 
by 
$$\var(t,s)=\f{\pa f}{\pa x_{n+2}} (t,u_0(t), \dots, u_0^{(n-1)}(t),s) + 
\sum_{m=1}^n (-1)^m h_{m,n-m+2}(t)-c_nt^n+\dots-c_1t-c_0.$$
Then

(i) $\var(t_0,u_0^{(n)}(t_0))=0.$

(ii) $\f{\pa\varphi}{\pa s}$ and $\f{\pa \var}{\pa t}$ exist and continuous.

(iii) $\f{\pa\var}{\pa s} (t_0,u_0^{(n)}(t_0))=\f{\pa^2f}{\pa x^2_{n+2}} 
(t_0,u_0(t_0),\dots, u_0^{(n-1)}(t_0) , u_0^{(n)}(t_0))\neq 0$.
\\
Therefore, Lemma 2.2 implies that 
\begin{multline*}
\exists \del_1,\del_2>0; \forall t\in (t_0-\del_1,t_0+\del_1)~ \exists ! 
z(t)\in (u_0^{(n)}(t_0)-\del_2, u_0^{(n)}(t_0)+\del_2);\\
\var(t,z(t))=0 \& z\in C^1(t_0-\del_1,t_0+\del_1)\& z(t_0)=u_0^{(n)}(t_0).
\end{multline*}
On the other hand, the continuity of $u_0^{(n)}$ implies that there is $\del_0 > 0$ such that
$$\forall t\in (t_0-\del_0,t_0+\del_0);~ 
u_0^{(n)}(t)\in (u_0^{(n)}(t_0)-\del_2, u_0^{(n)}(t_0)+\del_2).$$
Therefore, if $\del := \text{min}\lbrace \del_0, \del_1\rbrace$, for every $t\in (t_0-\del,t_0+\del)$
$$u_0^{(n)}(t)\in (u_0^{(n)}(t_0)-\del_2, u_0^{(n)}(t_0)+\del_2)$$
and
$$\exists ! 
z(t)\in (u_0^{(n)}(t_0)-\del_2, u_0^{(n)}(t_0)+\del_2);~
\var(t,z(t))=0.$$
Then, since $\var(t,u_0^{(n)}(t))=0$ for $t \in (a, b)$, we will have 
$$u_0^{(n)}(t)=z(t) \quad for \quad t\in (t_0-\del,t_0+\del)$$
and consequently , since $z\in C^1(t_0-\del_1,t_0+\del_1)$,
$$u_0^{(n)} \in C^1(t_0-\del,t_0+\del),$$
so
$$u_0\in C^{n+1} (t_0-\del,t_0+\del).$$
\end{proof}

\section{Regularity of the weak solution}
The following Lemmas are necessary for providing proof of Theorem 3.4.

\begin{lemma}
Let $\Omega$ be an open set in $\mathbb{R}$. Suppose $f:\Omega \times \mathbb{R}^n \rightarrow \mathbb{R}$ have the following properties:

(i) for all $(y_1, \dots, y_n) \in \mathbb{R}^n$ the function $x \mapsto f(x, y_1, \dots, y_n)$ is measurable on $\Omega$;

(ii) for a.a. $x \in \Omega$ the function $(y_1, \dots, y_n) \mapsto f(x, y_1, \dots, y_n)$ is continuous on $\mathbb{R}$.\\
If $\varphi_i: \Omega \rightarrow \mathbb{R}$ for $i = 1, \dots, n$ are (Lebesgue) measurable on $\Omega$, then
$$x \longmapsto f(x, \varphi_1(x), \dots, \varphi_n(x))$$
is a measurable function on $\Omega$.
\end{lemma}
\begin{proof}
See Remark 3.2.25 in [7].
\end{proof}
\begin{lemma}
If $g:=g(t,s)$ is a function from $[a,b]\ti \mathbb{R}$ to $\mathbb{R}$ such that 

(i) $\forall t\in [a,b] \; \exists !s(t)\in \mathbb{R}$; $g(t,s(t))=0$.

(ii) $\f{\pa g}{\pa s}>0$ on $[a,b]\ti \mathbb{R}$.

(iii) $g,\f{\pa g}{\pa s}$ are continuous on $[a,b]$.\\
Then, the function $t\lo s(t)$ is continuous on $[a,b]$. 
\end{lemma}
\begin{proof} 
We have the continuity on $(a,b)$ by the following fact which is second form of Implicit 
Function Theorem (see Remark 4.2.3 from [7]):
\paragraph{Fact:} If $\var:=\var(s,t)$ is a function from $[a,b]\ti \mathbb{R}$ to $\mathbb{R}$ 
such that 

(i) $\var(t_0,s_0)=0$.

(ii) $\f{\pa \var}{\pa s}(t_0,s_0)\neq 0$.

(iii) $\var,\f{\pa\var}{\pa s}$ are continuous in $t_0$.\\
Then,
$$\exists \del_1,\hat{\del}; \forall t\in (t_0-\del_1,t_0+\del_1) \exists ! 
z(t) \in (s_0-\hat{\del}, s_0+\hat{\del}); \var(t,z(t))=0,$$
the function $t\lo z(t)$ is continuous.\\
Moreover, it is continuous at the end points $a,~b$ by applying the following Fact (See Exercise 7.1.21 in [7]):\\
\\ 
Fact: {\textit {Let $g:[a, b] \times \mathbb{R} \rightarrow \mathbb{R}$ be a function and assume that for any $x \in [a, b]$ the
equation $g(x, z) = 0$ has a solution denoted by $z = z(x)$. If
$$\f{\pa g}{\pa z}(x,z)> 0~~~\text{on}~~[a,~b] \times \mathbb{R}$$
then this solution is unique. If, moreover, $g$ and $\f{\pa g}{\pa z}$ are continuous on
$[a, b] \times \mathbb{R}$, then $z = z(x)$ is continuous on $[a, b]$ as well.  }} 
\end{proof}

\begin{lemma}
Suppose $F$ be a functional from $W^{n,2}(a,b)$ to $\mathbb{R}$ and let $u_0\in 
C^n[a,b]$ be a local extremum of $F$. Then, $u_0$ is local extremum of 
$F|_{C^n[a,b]}$.
\end{lemma}
\begin{proof} 
For $u\in C^n[a,b]$, 
$$\|u^{(i)}\|_2^2 
= \int_a^b |u^{(i)}(x)|^2 dx \leq (b-a) \|u^{(i)}\|_\infty^2,$$
so
$$\|u^{(i)}\|_2 \leq \sqrt{b-a} \|u^{(i)}\|_\infty$$
and
$$\sum_{i=0}^n \|u^{(i)}\|_2 \leq \sqrt{b-a} \sum_{i=0}^n 
\|u^{(i)}\|_\infty.$$
Therefore,
$$\|u\|_{W^{n,2}(a,b)} \leq \sqrt{b-a} \|u\|_{C^n[a,b]}.$$
Now, by the following fact, proof is complete.
\paragraph{Fact:} Suppose $X,Y$ are normed spaces such that $Y\subs X$ and 
$\forall u \in Y$ $\|u\|_X <M\| u\|_Y$ for $M>0$.
Let $F$ be a functional from $X$ to $\mathbb{R}$ 
and $u_0\in Y$ be a local extremum of $F$. Then, $u_0$ is local extremum of 
$F|_Y$.
\end{proof}
The following regularity theorem is our major goal in this section.
\begin{theorem}
Suppose $n\geq 1$, $N,N'\subs
\{0,1,\dots,n\}$ and ${\mathcal{N}}_{N,N'}=\{u\in W^{n,2} (a,b): 
u^{(i)}(a)=u_i, u^{(j)}(b)=w_j\;\; \text{where}\;\; i\in N, j\in N'\}$.

Define the functional $F$ on ${\mathcal{N}}_{N,N'}$ by
\begin{align*} 
F(u)=\int_a^b f(t,u(t), u'(t),\dots,u^{(n)}(t))dt \tag{$3-1$}
\end{align*}
where $f=f(x_1,\dots,x_{n+2})$ is a function defined on $[a,b]\ti \mathbb{R}^n$ with 
continuous second partial derivatives with respect to all its variables. Let $h \in L_2(a, b)$, $c_1 \geq 0$ be
such that for a.a. $x_1 \in [a, b]$ and for all $(x_2, \dots, x_{n+2}) \in \mathbb{R}^{n+1}$,\\
\begin{align*}
&\vert f(x_1, x_2, \dots, x_{n+2}) \vert \leq h(x_1) + c_1(x^2_2 +\dots + x^2_{n+2}) \tag{$3-2$}\\
&\vert \f{\pa f}{\pa x_i}(x_1, x_2, \dots, x_{n+2}) \vert \leq h(x_1) + c_1(\vert x_2 \vert +\dots + \vert x_{n+2} \vert) ~\text{for}~ i \in \lbrace 2, \dots, n+2 \rbrace. \tag{$3-3$}
\end{align*}
Let $u_0\in{\mathcal{N}}_{N,N'}$ be a local extremum of $F$ with respect to ${\mathcal{N}}_{N,N'}$. 
For $t\in [a,b]$ and $s\in \mathbb{R}$ set 
$$\psi(t,s)=\f{\pa f}{\pa x_{n+2}}
(t,u_0(t),u'_0(t),\dots, u_0^{(n-1)}(t),s).$$
Assume that $\f{\pa\psi}{\pa s}>0$ on $[a,b]\ti \mathbb{R}$ and that for every fixed 
$t\in [a,b]$ the function $s\lo \psi(t,s)$ maps $\mathbb{R}$ onto $\mathbb{R}$. Then $u_0\in 
C^{n+1}[a,b]$.
\end{theorem}

\begin{proof}
First, it should be noticed that for $u \in \mathcal{N}_{N,N'}$ the function
$$t \longmapsto f(t,u(t),\dots, 
u^{(n)}(t))$$
is a measurable function on $(a, b) $ from Lemma 3.1. Moreover, for $u \in \mathcal{N}_{N,N'}$, by $(3-2)$ and the fact that $h, u, u^{\prime},\dots, u^{(n)} \in L^2$, we will have
\begin{align*}
\int_a^b f(t,u(t), u'(t),\dots,u^{(n)}(t))dt = &\int_a^b \vert f(t,u(t), u'(t),\dots,u^{(n)}(t)) \vert dt \\
\leq &\int_a^b h(t) dt + c_1 \left(\sum_{i=0}^{n} \int_a^b u^{(i)}(t) dt \right) \\
<&\infty.
\end{align*}
Then, for every $u \in \mathcal{N}_{N,N'},~F(u)<\infty$ and consequently $F$ is a well-defined function. Now, again, by Lemma 3.1, for every $i;~2\leq i \leq n+2$
$$t \longmapsto \f{\pa f}{\pa x_i} (t,u_0(t),\dots, 
u_0^{(n)}(t))$$
is a measurable functions on $(a, b)$. Moreover, by utilizing $(3-3)$ and Hölder's inequality, for every $v \in W$ we have
\begin{align*} 
\int_a^b \sum_{i=2}^{n+2} \f{\pa f}{\pa x_i} (t,u_0(t),\dots, u_0^{(n)}(t)) 
V^{(i-2)} (t)dt <\infty. \tag{$3-4$}
\end{align*}
Then, if we proceed literally as in the proof of Theorem 2.5 we arrive at the 
following equality which now holds for a.a. $t\in[a,b]$:
$$\f{\pa f}{\pa x_{n+2}} (t,u_0(t),\dots,u_0^{(n)}(t))+\sum_{j=1}^n (-1)^j 
h_{j,n+2-j}(t)-c_nt^n+\dots-c_1t-c_0=0.$$
We define function $g$ as follows: 
$$g(t,s)=\psi(t,s)+\sum_{j=1}^n (-1)^j h_{j,n+2-j}(t)-c_nt^n+\dots-c_1t-c_0.$$
For $\var_t(s):=\psi(t,s)$, we have
$$\var'_t(s)=\f{\pa \psi}{\pa s}(t,s) =\f{\pa^2f}{\pa x^2_{n+2}} 
(t,u_0(t),\dots, u_0^{(n-1)}(t),s)>0$$
therefore, $\var_t$ is one to one function. On the other hand by assumptions, 
$\var_t$ is surjective. Hence 
$$\forall t\in[a,b]\; ~\exists ! s(t)\in \mathbb{R};~ \var_t(s(t))=c_nt^n+\dots+c_1t+c_0-\sum_{j=1}^n (-1)^j 
h_{j,n+2-j}(t)$$
then
$$\forall t\in [a,b]\; ~\exists !s(t)\in \mathbb{R}; ~\psi(t,s(t))+\sum_{j=1}^n (-1)^j 
h_{j,n+2-j}(t) -c_nt^n+\dots-c_1t-c_0=0.$$
Consequently 
$$\forall t\in [a,b]\; ~\exists !s(t)\in \mathbb{R}; ~g(t,s(t))=0.$$
Then by
Lemma 3.2 the function $t\lo s(t)$ is continuous on $[a,b]$. On the other 
hand we have for every $t\in[a,b]$; $g(t,u_0^{(n)}(t))=0$. Therefore 
$$\forall t\in [a,b] \; u_0^{(n)}(t)=s(t)$$
and
$$u_0^{(n)} \;\; \text{is continuous}.$$
Hence $u_0\in C^n[a,b]$ and by Lemma 3.3 it is a local extremum of 
$F|_{C^n[a,b]}$. The assertion now follows from Theorem 2.5. 
\end{proof}

\begin{remark} 
It should also be mentioned that the growth conditions of (3-2) and (3-3) have been added to assumptions of Theorem 3.4 to guarantee the integrability of (3-1) and (3-4). Therefore, if in a problem we had this intention, the theorem can still be applied without these conditions be satisfied.     
\end{remark}

\begin{remark}
it should be taken into consideration that the differentiability condition of $f$ in Theorem 3.4 in some situations could be skipped. For instance, suppose, it is proved that that for every continuous $f$, the following differential equation have a weak solution: 
\begin{equation}
x^{\prime\prime}(t)=f(t,x(t));~~~t \in (0,1). \tag{$3-5$}
\end{equation}
However, $f$ is merely continuous and not differentiable, by utilizing Theorem 3.4 and in addition to considering that $\overline{C^{2}(X)}=C(X)$, it is again can be proved that the weak solution is in $C^2$. To illustrate, suppose $x_0$ is a weak solution of the differential equation $(3-5)$. For an arbitrary $n$, assume that $f_n$ is a function with differentiability conditions in Theorem 3.4 such that 
$$\Vert f_n - f \Vert_{\infty} < \f{1}{n}$$
Since $f_n$’s are continuous, the following equations 
$$x^{\prime\prime}(t)=f_n(t,x(t));~~~n \in \mathbb{N}$$ 
have weak solutions. On the other hand, since, now, $f_n$’s have the conditions of Theorem 3.4, then these solutions are in $C^2$; i.e. 
$$\exists x_n \in C^{2}(0, 1);~~ x_n^{\prime\prime}(t)=f_n(t,x_n(t)).$$ 
Now, we have
\begin{align*} 
\Vert x_n^{\prime\prime} - x_m^{\prime\prime} \Vert_{\infty} &= \text{Sup}_{t \in (0, 1)} \vert x_n^{\prime\prime}(t) - x_m^{\prime\prime}(t) \vert \\
&= \text{Sup}_{t \in (0, 1)} \vert f_n(t,x_n(t)) - f_m(t,x_m(t)) \vert\\
&\leqslant \Vert f_n - f_m \Vert_{\infty}.
\end{align*} 
So $\lbrace x_n^{\prime\prime} \rbrace_{n=1}^{\infty}$ is Cauchy in $C(0, 1)$. Therefore, 
$$\exists z \in C(0, 1);~~x_n^{\prime\prime}\rightarrow z~~\text{uniformly as}~~n \rightarrow \infty, $$ 
then 
$$\exists z \in C(0, 1);~~x_n^{\prime}\rightarrow \int_{0}^{t} z(s) ~ds~~\text{uniformly as}~~n \rightarrow \infty, $$ 
consequently 
\begin{equation}
\forall y \in C_{0}^{\infty}(0, 1)~~\int_{0}^{1} \left(x_n^{\prime}(t) - \int_{0}^{t} z(s) ds \right) y^{\prime}(t)dt~~\text{as}~~n \rightarrow \infty. \tag{$3-6$}
\end{equation}
On the other hand, since $x_\circ$ is a weak solution of the following equation 
$$x^{\prime\prime}(t)=f(t,x(t));~~t \in (0, 1)$$ 
we have 
\begin{equation}
\forall y \in C_{0}^{\infty}(0, 1)~~-\int_{0}^{1} x_\circ^{\prime}(t)  y^{\prime}(t)dt=\int_{0}^{1} f(t,x(t))  y(t)dt. \tag{$3-7$}
\end{equation} 
Moreover, as $x_n$ for every $n \in \mathbb{N}$ is a weak solution of the following equation: 
$$x^{\prime\prime}(t)=f_n(t,x(t));~~t \in (0, 1)$$ 
we have 
\begin{equation}
\forall y \in C_{0}^{\infty}(0, 1)~~-\int_{0}^{1} x_n^{\prime}(t)  y^{\prime}(t)dt=\int_{0}^{1} f_n(t,x(t))  y(t)dt. \tag{$3-8$}
\end{equation} 
Now, by $(3-7)$ and $(3-8)$ for every $y \in C_{0}^{\infty}(0, 1)$, 
\begin{align*}
\left \vert \int_{0}^{1} [x_{\circ}^{\prime}(t) - x_n^{\prime}(t)]  y^{\prime}(t)dt \right \vert&=\left \vert \int_{0}^{1} [ f(t,x_n(t)) - f_n(t,x_{\circ}(t)) ]  y(t)dt \right \vert\\
&\leqslant \Vert f_n - f \Vert_{\infty} \int_{0}^{1}  \vert y(t) \vert dt.
\end{align*} 
Consequently 
\begin{equation}
\forall y \in C_{0}^{\infty}(0, 1)~~\int_{0}^{1} \left(x_{\circ}^{\prime}(t) - x_n^{\prime}(t) \right) y^{\prime}(t)dt~~\text{as}~~n \rightarrow \infty. \tag{$3-9$}
\end{equation}
Then, by $(3-6)$ and $(3-9)$, it is concluded that 
$$\forall y \in C_{0}^{\infty}(0, 1)~~\int_{0}^{1} \left(x_{\circ}^{\prime}(t) - \int_{0}^{t} z(s) ds \right) y^{\prime}(t)dt=0$$
so, by Lemma 2.3, the following is resulted: 
$$x_{\circ}^{\prime}(t) = \int_{0}^{t} z(s) ds+c;~~z \in C(0, 1)$$ 
hence, 
$$x_\circ \in C^{2}(0, 1).$$ 
\end{remark}
\begin{example}
we will illustrate the application of Theorem 3.4 on the following Dirichlet boundary value problem 
\begin{equation*}
\begin{cases}
x^{(2n)}(t)+x^{\prime \prime}(t)+x^{3}(t)=f(t,x(t)), & t\in (0,1),\\
x(0)=x(1)=0,\\
\end{cases}\tag{$3-10$}
\end{equation*}
where $n \in \mathbb{N}$ and $f$  is a continuous function on $[0, 1] \times \mathbb{R}$. Put $H:=\{u\in W^{2n-1,2} (0,1): 
u^{(i)}(0)=0, u^{(j)}(1)=0\;\; \text{where}\;\; i\in N, j\in N'\}$ where $N=\lbrace 0 \rbrace $, $N^\prime=\lbrace 0\rbrace$. The functional
$$\psi(x):=\int_{0}^{1} \int_{0}^{x(t)} f(t,s) ds dt$$
defined on $H$ is of the class $C^1(H,\mathbb{R})$ and 
$$\psi^{\prime}(x)(h)=\int_{0}^{1}f(t,x(t)) h(t) dt, \quad x,h \in H.$$ 
Then 
$$F(x)=\int_{0}^{1} \left[\f {(-1)^n}{2}\vert x^{(n)}(t) \vert^{2}-\f {1}{2}\vert x^{\prime}(t) \vert^{2}+\f {1}{4}\vert x(t) \vert^{4}-\int_{0}^{x(t)} f(t,s) ds\right]dt$$ 
is of the class $C^1(H,\mathbb{R})$ and its critical points correspond to weak solutions of $(3-10)$. The regularity argument in Theorem 3.4 applied to $(3-10)$ implies that every weak solution is a classical solution in the sense that 
$$x \in C^{2n}_{0}[0,1]:=\{x\in C^{2n}[0,1]: x(0)= x(1)=0\}$$ 
and the equation in $(3-10)$ holds at every point $t$. Note that, in this example, the differentiability condition of $f$ was omitted based on Remark 3.6.    
\end{example}
\begin{example}
Many practical problems in applied sciences can be expressed as the following minimization problems (see [1, 3, 5, 6, 11], for example): 
\begin{equation*} 
\|x-x_\circ\|_{L_2(I)}^2+\lambda_1 \|x\|_{Y_1}+ \cdots +\lambda_n \|x\|_{Y_n}
\tag*{$(3-11)$}
\end{equation*}
where 
$$\|x-x_\circ\|_{L_2(I)}:=\Biggl(\int_I|x(t)-x_\circ(t)|^2dt\Biggr)^{\frac{1}{2}}$$ 
is the root-mean-square error (or more generally difference) between $x$ and $x_\circ$, and $\|x\|_{Y_i}$ for $i=1,\cdots, n$ are the norms of different smoothness spaces $Y_i$ for $i=1,\cdots,n$ respectively. $\lambda_i$ for $i=1,\cdots,n$ are parameters, if $\lambda_i$ is large, then necessarily $\|x\|_{Y_i}$ must be smaller at the minimum, i.e. $x$ must be smoother, while when $\lambda_i$ is small, $x$ can be rough, with $\|x\|_{Y_i}$ large. In the cases that $\|x\|_{Y_i}$ for $i=1,\cdots, n$ are in the following form
$$\|x\|_{Y_i}=\int_I f_i(t,x(t), x'(t),\dots,x^{(n)}(t))dt,$$ 
where $f_i=f_i(x_1,\dots,x_{n+2})$ is a function defined on $I \ti \mathbb{R}^n$ with 
continuous second partial derivatives with respect to all its variables, $(3-11)$ would be a problem in type of $(3-1)$. As a result, all papers that address discussions around the solutions of $(3-1)$, including this note, could possibly be important in investigating $(3-11)$. 
\end{example}
\section{Conclusion}
After three centuries, the study of the following problem 
$$F(x)=\int_a^b f(t,x(t),x'(t)) dt$$
and its variants still receive attention. Its applications are numerous in geometry and differential equations, in mechanics and physics, and in areas as diverse as engineering, medicine, economics, and renewable resources. However, in this paper, we discussed the generalization of this problem and focused on the regularity of its solutions in the hope that it would be useful to prove the regularity properties of problems’ solutions that arise from these disciplines. On the one hand, as we have already mentioned, in differential equations, addressing the notion of weak solution, a generalization of the notion of the classical solution, is beneficial because many nonlinear analysis methods are applicable to get a weak solution instead of a classical one. However, once we succeed in finding a weak solution, a inevitable question arises whether it has some better properties, e.g., the continuity of the first and second derivatives of the solution can be of interest. In fact, we generalized two theorems in the regularity theory which deals with these questions and is a very delicate issue in the theory of differential equations. On the other hand, optimal control problems with higher order are addressed more and more every year (see [19]; for example). Then, it is naturally more important to discuss the regularity properties of their solutions as well. Since the Hilbert-Weierstrass Theorem and Tonelli-Morrey Theorem are utilized in proving the regularity properties of optimal control problems (see chapter 23 in [4]), it is hoped that the availability of this article in hand can be an inspiration to prove the regularity of solutions to higher-order problems as well. Besides, it should be mentioned that some series of optimal control problems are equivalent to higher order variational problems (see [2]; for example). All in all, it is expected that this article would be of interest for all mathematicians who cherish Nonlinear Analysis and its history.


\end{document}